\documentclass[a4paper,11pt]{amsart}
\usepackage{graphicx}
\usepackage{amsmath}
\usepackage{amssymb}
\usepackage{oldgerm}
\usepackage[all]{xy}
\usepackage{tikz}
\usepackage{tikz-cd}

\newcommand{\Hom}{\operatorname{Hom}\nolimits}
\newcommand{\RHom}{\operatorname{RHom}\nolimits}

\newcommand{\id}{\operatorname{id}\nolimits}

\newtheorem{theo}{Theorem}[section]

\newtheorem{lemma}[theo]{Lemma}

\newcommand{\ten}{\otimes}
\newcommand{\lten}{\overset{\mathbf{L}}{\ten}}

\newcommand{\iso}{\stackrel{_\sim}{\rightarrow}}

\newcommand{\eps}{\varepsilon}

\begin{document}
\baselineskip=15pt
\title[Derived invariance of the cap product]{Derived invariance of the
cap product in Hochschild theory}
\author{Marco A.~Armenta A. and Bernhard Keller}
\address{M.~A.~: Centro de Investigaci\'on en Matem\'aticas A. C., Cubicle D104, 36240 Guanajuato, Gto. M\'exico}
\address{B.~K.~: Universit\'e Paris Diderot -- Paris~7, 
UFR de Math\'ematiques, Institut de
Math\'ematiques de Jussieu--PRG, UMR 7586 du CNRS, Case 7012, B\^atiment
Sophie Germain, 75205 Paris Cedex 13, France}

\email{drmarco@cimat.mx, bernhard.keller@imj-prg.fr}
%\date{\today}

 \subjclass[2000]{}

\begin{abstract} We prove derived invariance of the cap product for associative 
algebras projective over a commutative ring.

 \end{abstract}

\maketitle

%\tableofcontents

\section{Introduction}
It has been known since Rickard's work \cite{Rickard91} that Hochschild
cohomology is preserved under derived equivalence as a graded algebra
with the cup product. Using the methods of \cite{Rickard91}, one can also
show that Hochschild homology is preserved as a graded space, see
for example \cite{Zimmermann07}. Nevertheless, derived invariance of the cap product -- which provides an action of the Hochschild cohomology algebra on the Hochschild homology -- has not been considered. In this note, we prove that derived invariance holds as well for the cap product.

This paper is part of the Ph.~D.~thesis of the first author, whose advisors are Claude Cibils and Jos\'e Antonio de la Pe\~na, to whom he is very grateful. It enters into the
first author's project of showing the derived invariance of the Tamarkin-Tsygan
calculus associated with a $k$-projective algebra.

\section{Derived invariance}
Let $k$ be a commutative ring and $A$ an associative $k$-algebra, projective
as a $k$-module. We write $A^e$ for the envelopping algebra $A\ten_k A^{op}$.
We denote by $D(A)$ the unbounded derived category of the category of
right $A$-modules. For a bimodule $M$, we denote by $HH^\bullet(A,M)$ the
Hochschild cohomology with coefficients in $M$ and by $HH_\bullet(A,M)$ the
Hochschild homology with coefficients in $M$, see for example
\cite{CartanEilenberg56} or \cite{Weibel94}. We have canonical isomorphisms
\[
HH^n(A,M) \iso H^n(\RHom_{A^e}(A,M)) = \Hom_{D(A^e)}(A,M[n])
\]
and
\[
HH_n(A,M) \iso H_n(A\lten_{A^e} M).
\]
Let $f\in HH^m(A,A)$. The cap product by $f$ is a map
\[
f\cap ?\, : HH_n(A,M) \to HH_{n-m}(A,M).
\]
The following lemma gives an interpretation of the cap product in terms of the
derived category

\begin{lemma} The following square commutes, where the vertical arrows are the
canonical identifications.
\[
\xymatrix{ HH_m(A,M) \ar[d] \ar[rr]^{f\cap ?} &  & HH_{m-n}(A,M) \ar[d] \\
H_0(M\lten_{A^e} A[-m]) \ar[rr]_{H_0(\id\ten f)} & & H_0(M\lten_{A^e} A[n-m]).}
\]
\end{lemma}

\begin{proof} Let $Bar(A)$ be the bar resolution of $A$, we get
\[
\xymatrix{ M \lten_{A^e} A = Tot(M \ten_{A^e} Bar(A)) = M \ten_{A^e} Bar(A).}
\]
Let $x \in M$ and $y\in Bar(A)$, then
\[
\xymatrix{ H_0(\id \ten f)(\left[x \ten y\right]) = \left[ x \otimes f(y) \right] = f \cap [x \ten y].}
\]
\end{proof}

Now suppose that $A$ is derived equivalent to a $k$-projective algebra $B$.
By Rickard's Morita theory for derived categories \cite{Rickard89} \cite{Rickard91},
this implies that there exist bimodule complexes $X\in D(A^{op}\ten_k B)$ and
$X^\vee \in D(B^{op}\ten_k A)$ such that there are isomorphisms
$\eta: A \iso X \lten_B X^\vee$ and $\eps: X^\vee \lten_A X \iso B$ in 
$D(A^e)$ respectively $D(B^e)$. We may and will suppose that these
isomorphisms make the following triangles commutative:
\[
\xymatrix{X \ar[r]^-{\eta\ten X} \ar[rd]_{=} & X \lten_B X^\vee\lten_A X \ar[d]^{X \ten \eps} \\
  & X} \quad
\xymatrix{X^\vee \ar[r]^-{X^\vee \ten \eta} \ar[rd]_{=} & X^\vee\lten_A X \lten_B X^\vee 
\ar[d]^{\eps\ten X^\vee} \\
 & X^\vee.}
\]

As a consequence, the functor
\[
F = ?\lten_{A^e}(X\lten_k X^\vee) : D(A^e) \to D(B^e)
\]
is an equivalence with quasi-inverse $G=?\lten_{B^e} (X\lten_k X^\vee)$. 
We have canonical isomorphisms
\[
FA = A\lten_{A^e}(X\lten_k X^\vee) = X^\vee\lten_A A \lten_A X = X^\vee \lten_A X \iso B
\]
and
\[
GB= B \lten_{B^e}(X^\vee\lten_k X) = X \lten_B B \lten_B X^\vee = X \lten_B X^\vee \iso A.
\]
We obtain a canonical isomorphism
\begin{eqnarray*} 
HH^n(A,A)=\Hom_{D(A^e)}(A, A[n]) & \iso & \Hom_{D(B^e)}(X^\vee\lten_A X, X^\vee\lten_A X[n]) \\
& \iso & \Hom_{D(B^e)}(B, B[n]) =HH^n(B,B).
\end{eqnarray*}
By abuse of notation, we will still denote it by $f \mapsto Ff$.
Let
us suppose that $M$ is an $A$-bimodule such that
$N=FM$ is concentrated in degree $0$. For example, if $M=A$,
then $N=B$.

\begin{theo} There is a canonical isomorphism
\[
HH_\bullet(A,M) \iso HH_\bullet(B,N)
\]
such that for each $f\in HH^m(A,A)$ the following square commutes
\[
\xymatrix{ HH_n(A,M) \ar[rr]^{f\cap ?} \ar[d]_\cong & & HH_{n-m}(A,M) \ar[d]^\cong \\
HH_n(B,N) \ar[rr]_{Ff\cap?} & & HH_{n-m}(B,N).}
\]
\end{theo}

\begin{proof} 
We define the isomorphism
\[
HH_\bullet(A,M) \iso HH_\bullet(B,N)
\]
to be induced by the canonical chain of isomorphisms in $D(k)$
\[
M\lten_{A^e} A \iso M\lten_{A^e}(X\lten_B X^\vee) = M\lten_{A^e}(X\lten_k X^\vee) \lten_{B^e} B = FM \lten_{B^e} B = N \lten_{B^e} B.
\]
Let $f\in HH^m(A,A)$. It suffices to show that the following square is commutative
\[
\xymatrix{
M\lten_{A^e} A \ar[rr] \ar[d]_{M\ten f} & & M\lten_{A^e}(X\lten_k X^\vee)\lten_{B^e} B 
\ar[d]^{M\ten X \ten X^\vee \ten Ff}\\
M \lten_{A^e} A[m] \ar[rr] & & M\lten_{A^e}(X\lten_k X^\vee)\lten_{B^e} B[m] .
}
\]
This is implied by the commutativity of the square
\[
\xymatrix{
A \ar[d]_f \ar[rr] & & B \lten_{B^e}(X^\vee \lten_k X) \ar[d]^{(Ff)\ten X^\vee \ten X} \\
A[m] \ar[rr] & & B[m] \lten_{B^e} (X^\vee \lten_k X).
}
\]
In turn, this will follow from the commutativity of the square
\[
\xymatrix{
A \ar[d]_f \ar[rr] & & A \lten_{A^e}(X\lten_k X^\vee)\lten_{B^e}(X^\vee \lten_k X)
\ar[d]^{f\ten X \ten X^\vee \ten X^\vee \ten X}\\
A[m] \ar[rr] & & A[m] \lten_{A^e} (X\lten_k X^\vee)\lten_{B^e}(X^\vee \lten_k X).
}
\]
This last commutativity follows from the naturality of the adjunction
morphism $A \iso GFA$.
\end{proof}

\bibliographystyle{amsplain}
\bibliography{references}

\end{document}